\documentclass{amsart}
\usepackage{amssymb,euscript,amsmath, mathrsfs}
\usepackage[dvips]{graphicx}
\usepackage[dvips]{color}

\newcounter{ENUM}

\input xy
\xyoption{all}
\CompileMatrices

\def\<{\langle}
\def\>{\rangle}
\def\0{{{\bf 0}}}

\def\OO{{\mathcal O}}

\def\CF{{\mathcal F}}

\def\CH{{\mathcal H}}

\def\CT{{\mathcal T}}

\def\CX{{\mathcal X}}
\def\CY{{\mathcal Y}}

\def\AA{{\mathbb A}}
\def\CC{{\mathbb C}}

\def\FF{{\mathbb F}}
\def\GG{{\mathbb G}}

\def\PP{{\mathbb P}}
\def\QQ{{\mathbb Q}}
\def\RR{{\mathbb R}}
\def\TT{{\mathbb T}}

\def\tP{{\tilde P}}

\def\tSigma{{\tilde \Sigma}}

\newcommand{\trop}{\operatorname{Trop}}
\newcommand{\supp}{\operatorname{Supp}}
\newcommand{\Hilb}{\operatorname{Hilb}}

\newcommand{\ini}{\operatorname{in}}

\newcommand{\rig}{\mbox{\rm \tiny rig}}
\newcommand{\an}{\mbox{\rm \tiny an}}

\newcommand{\et}{\mbox{\rm \tiny \'et}}

\def\ZZ{{\mathbb Z}}

\def\R{{\mathbb R}}

\def\Hom{\operatorname{Hom}}

\def\sep{\operatorname{sep}}

\def\Gal{\operatorname{Gal}}

\def\Spec{\operatorname{Spec}}

\def\ord{\operatorname{ord}}
\def\val{\operatorname{val}}

\newcommand{\Span}{{\operatorname{Span}}}

\newcommand{\margh}[1]{}

\newtheorem{thm}{Theorem}[section]
\newtheorem{prop}[thm]{Proposition}
\newtheorem{lemma}[thm]{Lemma}
\newtheorem{cor}[thm]{Corollary}
\newtheorem{conj}[thm]{Conjecture}

\theoremstyle{definition}
\newtheorem{defn}[thm]{Definition}

\newtheorem{ex}[thm]{Example}

\newtheorem{rem}[thm]{Remark}

\numberwithin{equation}{section}

\usepackage[small,nohug]{diagrams}
\newarrow{Eq}	=====

\begin{document}
\title[Monodromy and Tropical Varieties]{Monodromy Filtrations and the Topology of Tropical Varieties}
\author{David Helm and Eric Katz}
\address{Department of Mathematics\\University of Texas at Austin\\1 University Station C1200\\Austin, TX 78712-0257}

\begin{abstract}
We study the topology of tropical varieties that arise from a certain natural 
class of varieties.
We use the theory of tropical degenerations to construct a natural, ``multiplicity-free'' parameterization
of $\trop(X)$ by a topological space $\Gamma_X$ and give a geometric interpretation
of the cohomology of $\Gamma_X$ in terms of the action of a monodromy operator on the
cohomology of $X$.  This gives bounds on the Betti numbers of $\Gamma_X$ in terms of the 
Betti numbers of $X$ which constrains the topology of $\trop(X)$.   We also obtain a description of the top power of the monodromy operator acting on middle cohomology of $X$ in terms of the volume pairing on $\Gamma_X$.
\end{abstract}

\maketitle

\section{Introduction}

Let $\OO$ be a discrete valuation ring with field of fractions $K$.
Tropicalization is a procedure which takes as input a $d$-dimensional subvariety of an algebraic torus over $K$,
$X\subset (K^*)^n$, and associates to it a balanced weighted rational 
$d$-dimensional polyhedral complex, $\trop(X)\subseteq\RR^n$.  
Several questions naturally arise in this framework; for instance, one may ask what combinatorial properties
of $\trop(X)$ correspond to geometric properties of $X$.
One may also ask what constraints being a tropicalization places on the topology of a polyhedral complex.  In 
\cite{Hacking}, Hacking proved that if $X$ is a subvariety of $(\CC^*)^n$ satisfying a certain 
genericity condition, then the link of the fan $\trop(X)$ only has reduced rational homology in 
the top dimension.  Hacking's result holds for a number of examples, including generic 
intersections of ample hypersurfaces in projective toric varieties.  In \cite[Sec. 10]{Sp3}, 
Speyer showed that if $C$ is a genus $g$ curve in $(K^*)^n$ satisfying a genericity condition 
then there exists a balanced metric graph $\Gamma$ with $b_1(\Gamma)\leq g$ and a parameterization 
$i:\Gamma\rightarrow \trop(C)$ that is affine-linear on edges.  Our results include an 
analogue of Hacking's result for varieties defined over $K$ or as a higher-dimensional 
generalization of Speyer's result.  One may consider the monodromy action of $\Gal(K^{\sep}/K)$
on the \'{e}tale cohomology $H^*_{\et}(X_{K^{\sep}}, \QQ_l)$ and ask what properties of the monodromy action are encoded in $\trop(X)$.  This is the algebraic analogue of the monodromy action of a family of varieties defined over a punctured disk.  The work of the second-named author with Hannah Markwig and Thomas Markwig \cite{KMM} relating the valuation of the $j$-invariant of a plane elliptic curve to the length of the cycle in its tropicalization can be seen in this light.  We give a generalization of that result.

All of our results require that the variety $X$ be {\em sch\"{o}n}, a natural condition introduced 
in \cite{Tevelev} and generalized by \cite{LQ} to the nonconstant coefficient case.
This condition means that the ambient torus $(K^*)^n$ may be compactified to a 
toric scheme $\PP$ over an extension of $\Spec \OO$ such that the intersection of $\CX=\overline{X}\subset \PP$ 
with each open torus orbit $U_P$ is smooth of the expected dimension.  For appropriate $\PP$, $\CX$ is then a simple normal
crossings degenerations of $X$ (c.f.~\ref{prop:degeneration}).   The construction of the desired toric scheme $\PP$ follows from Proposition \ref{prop:nc}, a technical result in polyhedral geometry which has not appeared, to our knowledge, in the literature before.

The existence of such simple normal crossings degenerations for a sch\"on $X$ allows us to construct 
a natural ``parameterizing space'' $\Gamma_X$.  This generalizes a construction introduced by
Speyer \cite{Sp3} when $X$ has dimension $1$.  The space $\Gamma_X$ is closely related to the
dual complex of an appropriate degeneration $\CX$ of $X$; in this guise it already appears implicitly
in \cite{Hacking}, as well as in \cite{HKT}.  Kontsevich-Soibelman use rigid analytic techniques
to construct a similar polyhedral complex, with an integral affine structure,
from a suitable degeneration of $X$ in~\cite{KS}.

The space $\Gamma_X$ we construct is {\em independent} of a choice of model $\CX$ for $X$;
it depends only on $X$ and its embedding in the torus.  Moreover, $\Gamma_X$ comes equipped with a canonical
map to $\trop(X)$.  A choice of sufficiently fine triangulation of $\trop(X)$ gives 
$\Gamma_X$ the structure of a polyhedral complex.  When $\Gamma_X$ is viewed in such a way,
the natural parameterization $\Gamma_X\rightarrow \trop(X)$ is affine-linear on polyhedra.  
This parameterization has several nice properties.
For instance, it is natural under monomial morphisms: if $X$ and $Y$ are sch\"{o}n 
subvarieties of tori $T$ and $T'$ and $\phi:T\rightarrow T'$ is a homomorphism taking $X$ to $Y$, then 
there is an induced map of complexes $\Gamma_X\rightarrow\Gamma_Y$ that commutes with parameterizations.  
Moreover, $\Gamma_X$ satisfies a balancing condition analogous to the one satisfied by all 
tropical varieties.  Finally, it is ``not far'' from $\trop(X)$: if the intersections of $\CX$
with open torus orbits $U_P$ in $\PP$ satisfy certain connectedness hypotheses,
we may equate the cohomology of $\Gamma_X$ and $\trop(X)$ in certain 
degrees.  We hope that parameterizing complexes will be seen as a 
fundamental object in tropical geometry.

Our main results (principally Theorem~\ref{thm:main}, Corollary~\ref{cor:CI}, and Proposition\ref{prop:npower}) relate
the cohomology of $\Gamma_X$ to geometric invariants of $X$.  In particular we consider the
{\'e}tale cohomology $H^*_{\et}(X_{K^{\sep}}, \QQ_l)$; this cohomology comes equipped with a
natural filtration, called the weight filtration.  We construct a natural isomorphism between the
``weight $0$'' subquotient $W_0 H^r_{\et}(X_{K^{\sep}}, \QQ_l)$ arising from this filtration
and the cohomology $H^r(\Gamma_X, \QQ_l)$ of $\Gamma_X$.  We then use this to show
that if $X$ is the generic intersection of ample hypersurfaces in a toric scheme $\PP$, 
then $H^r(\trop(X),\QQ_l)$ vanishes for $1\leq r<\dim X$, a non-constant coefficient analogue of 
Hacking's result.  We also prove results about the monodromy action on the middle cohomology $H^n_{\et}(X_{K^{\sep}}, \QQ_l)$.

The main tool we use is the Rapoport-Zink weight spectral sequence \cite{RZ}.  Under the sch\"{o}n 
condition, after a finite base-extension $\OO^{\prime}$ of $\OO$, we may compactify the ambient torus to 
a toric scheme $\PP$ defined over $\OO^{\prime}$ so that the central fiber of the closure
$\CX$ of $X$ in $\PP$ is a divisor with simple normal crossings.  The divisor, a degeneration of $X$, 
has a stratification coming from intersections of its irreducible components.  The Rapoport-Zink
spectral sequence then gives a very explicit formula for the cohomology on $X$, together with its
weight filtration, in terms of these strata.  The $E_1$-term 
of the weight spectral sequence is formed from the cohomology groups of the strata with 
boundary maps built from the data of restriction maps and Gysin maps.  The spectral sequence 
converges to the cohomology of the general fiber, and the induced filtration is the weight filtration.  
Moreover, the weight spectral sequence 
degenerates at $E_2$.  We thus obtain an explicit description of the smallest nontrivial piece of 
the filtration which is isomorphic to the cohomology groups of $\Gamma_X$.

It is interesting to compare this result to results of Berkovich \cite{B} on rigid analytic spaces.
In particular, Berkovich shows that the cohomology group 
$H^r_{\et}(\overline{X}_{K^{\sep}},\QQ_l)_{-r}$ arising in our result is isomorphic to
the cohomology of the Berkovich space $X^{\an}$ attached to $X$.  Our result thus suggests 
a strong link between $\Gamma_X$ and $X^{\an}$.  In fact, Speyer \cite{Speyer} constructs a
natural map from $X^{\an}$ to $\trop(X)$.  This map factors through the map $\Gamma_X \rightarrow \trop(X)$,
and it is natural to ask if the resulting map $X^{\an} \rightarrow \Gamma_X$ map is a homotopy 
equivalence.  Links between tropical geometry and rigid
geometry have also appeared in works of Einsiedler-Kapranov-Lind~\cite{EKL}, and Payne~\cite{P}.

Under additional hypotheses, one can relate the results above to questions involving monodromy.
A variety defined over $\Spec K$ is analogous to a family of varieties defined 
over a punctured disk.  
The fundamental group of the punctured disk acts on the cohomology of a general fiber of such a
family by monodromy.  The analogue of this monodromy action for varieties over $\Spec K$ is the
action of the inertia group $I_K$ of $K$ on the \'{e}tale cohomology 
$H^*_{\et}(X_{K^{\sep}},\QQ_l)$.  After a 
possible finite base-extension of $\OO$, this action is unipotent, and is 
given by the {\em monodromy operator}
$$N: H^i_{\et}(X_{K^{\sep}},\QQ_l) \rightarrow H^i_{\et}(X_{K^{\sep}},\QQ_l(-1))$$
an endomorphism of the \'etale cohomology that is essentially the (matrix) logarithm of
the action.  (We refer the reader to section~\ref{sec:monodromy} for precise
definitions).  The action of $N$ induces an increasing filtration on the cohomology.
The weight-monodromy conjecture asserts that this filtration coincides (up to a shift in degree) 
with the weight filtration
described above.  Although it is not completely settled, this conjecture is known to be
true in many cases of interest; for instance, it is known $X$ is a curve, surface, or
an abelian variety.  Ito \cite{Ito} has proven the weight-mondromy conjecture when $\OO$ has
equal characteristic.  Thus, in these situations, one can interpret Theorem~\ref{thm:main}
as an isomorphism between the cohomology of $\Gamma_X$ and the smallest
nontrivial piece of the monodromy filtration of the cohomology of $\overline{X}_K$, the closure 
of $X_K$ in the generic fiber of $\PP$:
$$H^r(\Gamma_X, \QQ_l) \cong H^r_{\et}(\overline{X}_{K^{\sep}},\QQ_l)_{-r}.$$
As a consequence, Corollary \ref{cor:bound} bounds 
the Betti numbers of $\Gamma_X$ in terms of those of $X$:
$$b_r(\Gamma_X) \leq \frac{1}{r+1} b_r(X)$$
generalization Speyer's result on curves.
  
We apply our techniques to get a description of some of the monodromy action and not just of the monodromy filtration.  In Proposition \ref{prop:npower}, we give an interpretation of the top power of monodromy operator acting on the middle-dimensional cohomology
\[N^d:H^d_{\et}(X_{K^{\sep}},\QQ_l)_{d}/H^d_{\et}(X_{K^{\sep}},\QQ_l)_{d-1}\rightarrow H^d_{\et}(X_{K^{\sep}},\QQ_l)_{-d}(-d)\]
using the isomorphisms 
\begin{eqnarray*}
H^d_{\et}(X_{K^{\sep}},\QQ_l)_{d}/H^d_{\et}(X_{K^{\sep}},\QQ_l)_{d-1}&\cong&H_d(\Gamma,\QQ_l)(-d)\\ 
H^d_{\et}(\overline{X}_{K^{\sep}},\QQ_l)_{-d}&\cong&H^d(\Gamma_X, \QQ_l).
\end{eqnarray*}
The operator can be viewed as a bilinear pairing on $H_d(\Gamma_X, \QQ_l)$ in which case it is the volume pairing that takes a pair of top-dimensional cycles to the oriented volume of their intersection.  This specializes to the length pairing in the case of curves.  In the case of genus $1$ curves, by a straightforward application of the conductor-discriminant formula, we are able to recover the spatial sch\"on analogue of a result proved by the second-named author with H. Markwig andT. Markwig \cite{KMM}: the valuation of the $j$-invariant of an elliptic curve $X$ with potentially multiplicative reduction is equal to $-a$ where $a$ is the length of the unique cycle in $\Gamma_X$.  

Our techniques are very similar to those of Hacking and Speyer.  Hacking uses a spectral sequence coming 
from a weight filtration on a complex of differential forms while we use the Rapoport-Zink spectral 
sequence.  Speyer's results use a resolution of the structure sheaf of a degeneration of a curve $C$ coming 
from a stratification induced by a toric scheme while we use a locally constant sheaf.

We should mention some related results.  In \cite{GS}, Gross and Siebert construct 
a scheme $X_0$ from an integer affine manifold and a toric polyhedral decomposition.  If $X_0$ is 
embedded in a family $\CX$ over $\CC[[t]]$, they are able to determine the limit mixed Hodge structure 
in terms of the combinatorial data.  In \cite{TMNF}, the second-named author and Stapledon give a description of the Hodge-Deligne polynomial of the limit mixed Hodge structure of a sch\"{o}n family of varieties over a punctured disk in terms of $\trop(X)$ and initial degenerations of $X$.

We would like to thank Brian Conrad, Richard Hain, Kalle Karu, Sean Keel, Sam Payne, Zhenhua Qu, Bernd Siebert, Martin Sombra, 
David Speyer, Alan Stapledon, and Bernd Sturmfels for valuable discussions.

\section{Toric Schemes}

We begin by reviewing a construction that attaches a degenerating family of
toric varieties over a discrete valuation ring to a rational polyhedral complex in $\RR^n$.
This has appeared several times in the literature~\cite{Speyer}~\cite{NS}
~\cite{Smirnov}.  We follow the approach of~\cite{NS} here.  Fix a discrete
valuation ring $\OO$, with field of fractions $K$ and residue field $k$, and 
a uniformizer $\pi$ of $\OO$.

\begin{defn} A rational polyhedral complex in $\RR^n$ is a collection 
$\Sigma$ of finitely many convex rational polyhedra
$P \subset \RR^n$ with the following properties:
\begin{itemize}
\item If $P \in \Sigma$ and $P^{\prime}$ is a face of $P$, then $P^{\prime}$
is in $\Sigma$.
\item If $P,P^{\prime} \in \Sigma$ then $P \cap P^{\prime}$ is a face of both
$P$ and $P^{\prime}$.
\end{itemize}
If, in addition, the union $\bigcup_{P\in\Xi} P$  is equal to $\R^n$, then $\Sigma$ is said to be {\em complete}.
\end{defn}

Given a $\Sigma$ as above, we can construct a fan $\tSigma$ in 
$\RR^n \times \RR_{\geq 0}$ as follows: for each $P \in \Sigma$ let
$\tP$ be the closure in $\RR^n \times \RR_{\geq 0}$ of the set
$$\{(x,a) \subset \RR^n \times \RR_{>0} : \frac{x}{a} \in P\}.$$
Then $\tP$ is a rational polyhedral cone in $\RR^n \times \RR_{>0}$.  Its 
facets come in two types:
\begin{itemize}
\item cones of the form $\tP^{\prime}$, where $P^{\prime}$ is a facet of $P$, and
\item the cone $P_0 = \tP \cap (\RR^n \times \{0\})$, which is the limit as $a$ goes to zero
of the polyhedron $aP$ in $\RR^n$.
\end{itemize}
We let $\tSigma$ be the collection of cones of the form $\tP$ and $P_0$ for $P$ in $\Sigma$.
If $\Sigma$ is a subcomplex of a complete rational polyhedral complex, we say that $\Sigma$ is {\em completable} and, by Corollary 3.12 of \cite{BGS}, $\tSigma$ is a rational polyhedral fan in $\RR^n \times \RR_{\geq 0}$.   Note that
$\Sigma = \tSigma \cap (\RR^n \times \{1\})$.  On the other hand
the fan $\Sigma_0$ given by $\tSigma \cap (\RR^n \times \{0\})$ is the limit as
$a$ approaches zero of the polyhedral complexes $a\Sigma$.

\begin{rem} \rm In fact, by results of \cite{BGS}, the association $\Sigma \mapsto \tSigma$ defines a bijection
between the set of complete polyhedral complexes in $\RR^n$ and the set of complete fans in 
$\RR^n \times \RR_{\geq 0}$ for which every cone contained in $\RR^n \times \{0\}$
is the boundary of a cone that meets $\RR^N \times \RR_{>0}$.  Its inverse is
$\tSigma \mapsto \tSigma \cap (\RR^n \times \{1\}).$
\end{rem}

Let $X(\tSigma)_{\ZZ}$ be the toric scheme over $\ZZ$ associated to the fan $\tSigma$.  (The
construction associating a toric variety to a fan is usually given over a field, but works
just as well with coefficients in $\ZZ$.)
Projection from $\RR^n \times \RR_{\geq 0}$ to $\RR_{\geq 0}$ induces a map of 
fans from $\tSigma$ to the fan
$\{0,\RR_{\geq 0}\}$ associated to $\AA^1_{\ZZ}$.  This gives rise to a map of
toric varieties $\pi_{\ZZ}: X(\tSigma)_{\ZZ} \rightarrow \AA^1_{\ZZ}$.  As remarked in~\cite{NS},
this map is flat and torus equivariant.  Let $\iota: \Spec \OO \rightarrow \AA^1_{\ZZ}$
be the map corresponding to the map $\ZZ[t] \rightarrow \OO$ that takes $t$ to $\pi$. 
We let $X(\tSigma)$ be the scheme over $\OO$ obtained by base change from $X(\tSigma)_{\ZZ}$
via $\iota$, and let $\pi: X(\tSigma) \rightarrow \OO$ be the projection.

We summarize results of~\cite{NS} concerning this construction:
\begin{itemize}
\item The general fiber $X(\tSigma) \times_{\Spec \OO} \Spec K$ is isomorphic
to the toric variety over $K$ associated to $\Sigma_0$.
\item If $\Sigma$ is {\em integral}, i.e. the vertices of every polyhedron in $\Sigma$
lie in $\ZZ^n$, then the special fiber $X(\tSigma)_k = X(\tSigma) \times_{\Spec \OO} \Spec k$
is reduced.   
\item There is an inclusion-reversing bijection
between closed torus orbits in $X(\tSigma)_k$ and polyhedra $P$ in $\Sigma$; the irreducible 
components of $X(\tSigma)_k$ correspond to vertices in $\Sigma$; the intersection of a collection
of irreducible components corresponds to the smallest polyhedron in $\Sigma$ containing all
of their vertices.
\end{itemize}

Note that adjoining a $d$th root of $\pi$ to $\OO$ has the effect
of rescaling $\Sigma$ by $d$; that is, if $\OO^{\prime} = \OO[\pi^{\frac{1}{d}}]$, then
the base change of the family $X(\tSigma) \rightarrow \OO$ is the family 
$X(\widetilde{(d\Sigma)}) \rightarrow \OO^{\prime}$.  
In particular,
given any toric scheme coming from a polyhedral complex $\Sigma$, we can choose $d$
such that $d\Sigma$ is integral; after taking a suitable ramified base change of $\OO$
the special fiber of the family $X(\tSigma) \rightarrow \OO$ will be reduced.

We will be particularly interested in degenerations of toric varieties in which the special
fiber is a divisor with simple normal crossings.  These are easy to construct, because the
boundary of a smooth toric variety is always a divisor with simple normal crossings.  We make use of the following result in polyhedral combinatorics:

\begin{prop} \label{prop:nc}
Let $\Sigma$ be a complete rational polyhedral complex in $\RR^n$. 
There exists an integer $d$, and a subdivision $\Sigma^{\prime}$ of $d\Sigma$ such that the general
fiber of the scheme $X(\tSigma^{\prime})$ is a smooth toric variety
and the special fiber of $X(\tSigma^{\prime})$ is a divisor with simple normal crossings.  Moreover, if the recession fan $\Sigma_0$ is already simplicial and unimodular, $\Sigma^{\prime}$ can be chosen to have $\Sigma_0'=\Sigma_0$.
\end{prop}
\begin{proof} Choose an integer $l_1$ such that $l_1\Sigma$ is integral.
Fulton~\cite[Sec. 2.6]{Fulton} gives an algorithm for constructing a subdivision
$\tSigma^{\prime}$ of the fan $\widetilde{l_1\Sigma}$ such that all the cones of $\tSigma_1$
are simplicial and unimodular.  Pick an integer $l_2$ sufficiently divisible so $\Sigma_1=\tSigma^{\prime} \cap (\RR^n \times \{l_2\})$ is integral.  $\Sigma_1$ is a subdivision of $\RR^n$ with the property that each of its recession cones is simplicial.  Because $\tSigma_1$ is simplicial each of its cones is of the form $\tilde{P}+Q_0$ where $P$ is a 
simplex in $\Sigma_1$, $Q_0$ is a cone in the recession fan $(\Sigma_1)_0$, and $\tilde{P}\cap Q_0=\{0\}$.    Consequently, the corresponding polyhedron of $\Sigma_1$ is $(\tilde{P}+Q_0)\cap (\RR^n\times\{l_2\})=P+Q_0$.
For a polyhedron $F$ in $\RR^k$, let $N_F=\ZZ^k\cap \Span_\RR(F-w)$ 
where $w$ is a point of $F$.  If $F$ is a rational polytope, then $N_F$ has the property a basis of it can be extended to a basis of $\ZZ^k$.

We claim $N_P+N_{Q_0}=N_{P+Q_0}$ for every cone $\tilde{P}+Q_0$ of $\tSigma_1$.  It is clear that $N_P+N_{Q_0}\subseteq N_{P+Q_0}$.  The inclusion $\ZZ^n\hookrightarrow\ZZ^{n+1}=\ZZ^n\times\ZZ$ given by $x\mapsto (x,l_2)$ identifies $N_{P+Q_0}$ with the intersection of $N_{\tilde{P}+Q_0}$ with $\ZZ^n\times \{l_2\}$.   Since $\tilde{P}+Q_0$ is unimodular, 
$N_{\tilde{P}}+N_{Q_0}$ is equal to $N_{\tilde{P}+Q_0}$.    Therefore, if $x\in N_{P+Q_0}$, $(x,l_2)=x_{\tilde{P}}+x_{Q_0}$ where $x_{\tilde{P}}\in \tilde{P}$ and $x_{Q_0}\in Q_0$.  Since the last coordinate of $x_{Q_0}$ is $0$, $x_{\tilde{P}}=(x_P,l_2)$ for $x_P\in N_P$.  It follows that $x=x_P+x_{Q_0}$.

Let $\Sigma_1^b$ be the union of the bounded polyhedra of $\Sigma_1$. 
By an important step of the proof of semi-stable reduction \cite[Ch. 3, Thm 4.1]{KKMS}, there is an integer $l_2$ and a unimodular triangulation ${\Sigma_1^b}'$ of $l_2\Sigma_1^b$.  This induces a subdivision of $l_2\Sigma_1$ where the polyhedra whose relative interior is contained in the relative interior of $l_2(P+Q_0)$  are of the form $P'+Q_0$ where $P'$ is a simplex in ${\Sigma_1^b}'$ whose relative interior is contained in the the relative interior of $l_2P$.  We call this subdivision $\Sigma'$.  It is simplicial by construction.  We claim that it is also unimodular.   It suffices to show that maximal cones in $\tSigma'$ are unimodular.  Let $\tilde{P}'+Q_0$ be a  maximal cone in $\tSigma'$.  Then the relative interior of $P'$ is contained in the relative interior of $l_2P$ with $\dim P=\dim P'$.   Since $P'$ is unimodular, its $\ZZ$-affine span is $N_{P'}$.    Consequently, since $N_{P'}+N_{Q_0}=N_P+N_{Q_0}=N_{P+Q_0}$, we see that any element of $N_{P+Q_0}$ can be written as integer combination $\sum m_i v_i+\sum n_iw_i$ where $v_i$ are vertices of $P'$, $w_i$ are the primitive vectors along the rays of $Q_0$, and $\sum m_i=1$.  Consequently, any element of $N_{P+Q_0}\times \{1\}$ can be written as  an integer combination of the primitive vectors along the rays of $\tilde{P'}+Q_0$.  Consequently these vectors generate $N_{\tilde{P}+Q_0}\subset\ZZ^{n+1}$.  Therefore $\tilde{P}'+Q_0$ is smooth.

If $\Sigma_0$ was simplicial and unimodular to begin with, none of these steps would have affected the cones $Q_0$ of $\Sigma_0$

The upshot is that $X(\tSigma^{\prime})$ is a smooth toric variety
with a birational morphism $X(\tSigma^{\prime}) \rightarrow X(\tSigma)$.  The induced map
$X(\tSigma^{\prime}) \rightarrow \OO$ is the toric scheme associated to
the integral polyhedral complex $\Sigma^{\prime}$.
The general fiber of $X(\tSigma^{\prime})$ over $\OO$ corresponds to the fan 
$\Sigma^{\prime}_0$, and is therefore smooth.  The special fiber is a union of
irreducible components of the boundary of $X(\tSigma^{\prime})$ and is therefore
a divisor of $X(\tSigma^{\prime})$ with simple normal crossings.
\end{proof}

\section{Tropical Degenerations}

We now describe the applications of tropical geometry to the study of degenerations of varieties over $K$.
These techniques have their origins in the Speyer's thesis \cite{Speyer}.  
The approach we take here is due to Tevelev \cite{Tevelev} in the ``constant coefficient case''; 
the extension of Tevelev's work to the case of an arbitrary DVR done by Luxton and Qu \cite{LQ}.

Let $\overline{K}$ be an algebraic closure of $K$.  There is a unique valuation 
$$\ord: \overline{K} \rightarrow \QQ$$
such that $\ord(\pi) = 1$.  

Let $\CT \cong \GG_m^n$ be a split $n$-dimensional torus over $\OO$, and let
$T = \CT \times_{\OO} K$ be the corresponding torus over $K$.  
The valuation $\ord$
gives rise to a map
$$\val: \CT(\overline{K}) \rightarrow \QQ^n,$$
by fixing an isomorphism of $\CT$ with $\GG_m^n$ (and hence an isomorphism of $\CT(\overline{K})$ 
with $(\overline{K}^*)^n$.)
Let $X$ be a closed subvariety of $T$, defined over $K$.

\begin{defn}[\cite{EKL}, 1.2.1]: The tropical variety $\trop(X)$ associated to $X$ is
the closure of $\val(X(\overline{K}))$ in $\RR^n$.
\end{defn}

Given such an $X$, one can ask for a well-behaved compactification $\overline{X}$ of $X$,
and a well-behaved degeneration of $\overline{X}$ over $\OO$.  The problem of
finding such a degeneration is intimately connected to the set $\trop(X)$.

Let $\Sigma$ be a completable rational polyhedral complex in $\RR^n$, and
let $\PP$ be the corresponding toric scheme over $\OO$.  Identify
the group of cocharacters of $\CT$ with $\ZZ^n$ in $\RR^n$; this identifies $\TT$ with 
the open torus orbit on $\PP$.

We can thus take the closure of $\CX$ of $X$ in $\PP$.  By \cite{Speyer}, 2.4.1, 
the scheme $\CX$ is proper over $\OO$ if, and only if, $\supp \Sigma$ contains
$\trop(X)$.  We assume henceforth that $\supp \Sigma$ contains $\trop(X)$.  Let $\overline{X}$ be the fiber
of $\CX$ over $K$, and $\CX_k$ be the special fiber of $\CX$.  The natural multiplication map
$$\CT \times_{\OO} \PP \rightarrow \PP$$ 
restricts to a multiplication map
$$m: \CT \times_{\OO} \CX \rightarrow \PP.$$ 

\begin{defn} The pair $(X,\PP)$ is {\em tropical} if the map
$$m: \CT \times_{\OO} \CX \rightarrow \PP$$
is faithfully flat, and $\CX \rightarrow \OO$ is proper.
\end{defn}

We then have the following, due to Tevelev \cite{Tevelev} in the constant coefficient case.
In the non-constant coefficient case they can be found in \cite{LQ}.

\begin{prop} A subvariety $X\subset(\overline{K}^*)^n$ admits a tropical pair $(X,\PP)$.  
\end{prop}

\begin{prop}
Suppose $(X,\PP)$ is tropical and let
$\PP^{\prime} \rightarrow \PP$ be a morphism of toric schemes corresponding
to a refinement $\Sigma^{\prime}$ of $\Sigma$.  Then $(X,\PP^{\prime})$ is also tropical. 
\end{prop}

\begin{prop} If $(X,\PP)$ is a tropical pair then $\supp \Sigma = \trop(X)$.
\end{prop}

Following Speyer (\cite{Speyer}, 2.4) If $(X,\PP)$ is a tropical pair, we call $\overline{X}$ a 
{\em tropical compactification} of $X$, and $\CX_k$ a {\em tropical degeneration} of $X$.

The combinatorics of the special fiber of a tropical degeneration of $X$ is closely
related to the combinatorics of $\trop(X).$  In particular if $(X,\PP)$ is a tropical
pair, and $\CX$ is the corresponding tropical degeneration, then a polyhedron $P$
of $\Sigma$ corresponds to the closure of a torus orbit in the special fiber of $\PP$.  
Call this torus orbit closure $\PP_P$.  Then the intersection $\CX_P$ of $\CX$ with $\PP_P$ 
is a closed subscheme of $\CX_k$.  Moreover, if $P$ and $P^{\prime}$ are
polyhedra of $\Sigma$, and $Q$ is the smallest polyhedron in $\Sigma$
containing both $P$ and $P^{\prime}$, then the intersection of $\CX_P$ and $\CX_{P^{\prime}}$
is $\CX_Q$.

Let $U_P$ be the open torus orbit corresponding to $P$.
Fix a point $p$ in $U_P$, and consider the fiber over $p$
of the multiplication map
$$m: \CT \times_{\OO} \CX \rightarrow \PP.$$
On the one hand, $m^{-1}(p)$ is nonempty of dimension equal to the dimension of $X$,
since $m$ is flat and surjective.  On the other hand, by projection onto $\CX$, $m^{-1}(p)$ 
is isomorphic to the product $\CT_P\times (\CX \cap U_P)$, where $\CT_P$ is the subgroup 
of $\CT_0$ that acts trivially on $U_P$.  Since $(\CX \cap U_P)$ is dense in $\CX_P$ we find that
$\CX_P$ is nonempty of dimension equal to the dimension of $X$ minus the dimension of $P$.

On the other hand, let $w$ be a point in the relative interior of $P$.  Then $w$ corresponds
to a cocharacter of $T$, and $w(\pi)$ specializes to a point $p$ in $U_P$.   Projection
onto $\CT$ identifies $m^{-1}(p)$ with the mod $\pi$ reduction $\ini_w X$ of $w(\pi)X$.
(More formally, $\ini_w X$ can be defined as the special fiber of the closure in $\CT$
of the subscheme $w(\pi)X$ of $T$.  Note that this depends only on $X$ and $T$, not on our choice of
$\Sigma$.) In particular we have 
$$\ini_w X \cong \CT_P \times (\CX \cap U_P)$$ 
for any $w$ in the relative interior of $P$.  We have thus shown:

\begin{lemma} \label{lemma:comp}
The space $\ini_w X$ is a torus bundle over $\CX \cap U_P$.  In particular, if
$C(\ini_w X)$ is the set of connected components of $(\ini_w X)_{\overline{k}}$,
and $C(\CX \cap U_P)$ is the set of connected components of $(\CX \cap U_P)_{\overline{k}}$,
then the maps
$$\ini_w X \cong m^{-1}(p) \rightarrow \CX \cap U_P$$  
give a natural bijection of $C(\ini_w X)$ with $C(\CX \cap U_P)$.
\end{lemma}

We will be particularly interested in tropical pairs $(X,\PP)$ where the multiplication
map $m: \CT \times_{\OO} \CX \rightarrow \PP$ is {\em smooth}.  This condition is
due to Tevelev.

\begin{defn} A subvariety $X$ of $\CT$ is {\em sch\"on}
if there exists a tropical pair $(X,\PP)$ such that the multiplication map
$$m: \CT \times_{\OO} \CX \rightarrow \PP$$ is smooth.
\end{defn}

One then has (\cite{LQ}, 6.7):
\begin{prop} If $X$ is sch\"on, then for {\em any} tropical pair
$(X,\PP)$, the multiplication map
$$m: \CT \times_{\OO} \CX \rightarrow \PP$$ 
is smooth.
\end{prop}

Note that if $X$ is sch\"on then it is smooth (consider the preimage of the identity in T
under the multiplication map.)  In fact, we have:

\begin{prop} \label{prop:schon} The following are equivalent:
\begin{enumerate}
\item $X$ is sch\"on.
\item $\ini_w X$ is smooth for all $w \in \trop(X)$.
\item For any tropical pair $(X,\PP)$, and any polyhedron $P$ in $\Sigma$,
$\CX \cap U_P$ is smooth.
\end{enumerate}
\end{prop}
\begin{proof}
Statements 2) and 3) are clearly equivalent since we have seen that $\ini_w X$ is the product
of a torus with $\CX \cap U_P$, where $P$ is the polyhedron in $\Sigma$ that contains $w$ in its
relative interior.

As for the equivalence of 1) and 2), fix a tropical pair $(X,\PP)$.  We have seen that the 
fibers of the multiplication map
$$m: \CT \times_{\OO} \CX \rightarrow \PP$$ 
are isomorphic to $\ini_w X$ as $w$ ranges over $\trop(X)$.  So 1) implies 2) is clear.
For the converse, note that since $m$ is faithfully
flat, to show it is smooth it suffices to show that it has smooth fibers.  
\end{proof}

It is easy to construct tropical degenerations of sch\"on varieties $X$ in which the special
fiber is a divisor with simple normal crossings.  In particular we have:

\begin{prop}[c.f.~\cite{Hacking}, proof of 2.4] \label{prop:degeneration}
Let $X$ be sch\"on.  There exists an integer $d$ and a tropical pair $(X,\PP)$ over 
$\OO[\pi^{\frac{1}{d}}]$ such that
$\overline{X}$ is smooth over $K[\pi^{\frac{1}{d}}]$, and $\CX_k$ is a divisor in $\CX$ with 
simple normal crossings.  
\end{prop}
\begin{proof}
Let $(X,\PP)$ be any tropical pair over $\OO$, and let $\Sigma$ be the rational polyhedral
complex corresponding to $\PP$.  By Proposition~\ref{prop:nc}, we can
find a refinement $\Sigma^{\prime}$ of $\Sigma$ and an integer $d$ such that the corresponding toric 
scheme $\PP^{\prime}$ (viewed over $\OO[\pi^{\frac{1}{d}}]$)
has smooth general fiber, and special fiber a divisor with simple normal crossings.

Then $(X,\PP^{\prime})$ is also tropical, and the multiplication map
$$m: \CT \times_{\OO} \CX^{\prime} \rightarrow \PP^{\prime}$$
is smooth by the previous proposition.  Since the special fiber of $\PP^{\prime}$ is a divisor
with simple normal crossings, so is the special fiber of $\CT \times_{\OO} \CX^{\prime}$.
Hence the special fiber of $\CX^{\prime}$ is a divisor with simple normal crossings as well.
Similarly, the general fiber of $\CX^{\prime}$ is smooth because the general fiber of
$\PP^{\prime}$ is smooth.
\end{proof}

\begin{defn} We call a pair $(X,\PP)$ of the sort produced by Proposition~\ref{prop:degeneration}
a {\em normal crossings pair}.  If $(X,\PP)$ is a normal crossings pair, and $\Sigma$
is the polyhedral decomposition of $\trop(X)$ corresponding to $\PP$, we say that
$\Sigma$ is a {\em normal crossings decomposition} of $\trop(X)$.
\end{defn}

\begin{rem} \rm In much of what follows, we will often need to attach a normal
crossings pair to a sch\"on variety $X$ over $\OO$.  To do this we may need to replace
$\OO$ with a ramified extension $\OO[\pi^{\frac{1}{d}}]$; this is harmless and we often
do so without comment.
\end{rem}

\section{Parameterized Tropical Varieties} \label{sec:param}

In this section, given a sch\"on subvariety $X$ of $\CT$, we construct a
natural parameterization of $\trop(X)$ by a topological space $\Gamma_X$.  This
parameterization is functorial in a sense we make precise below.  Moreover, we
will see in the next section that the space $\Gamma_X$ encodes more precise information
about the cohomology of $X$ than $\trop(X)$ does.  Our approach generalizes a construction
of Speyer (\cite{Sp3}, proof of Theorem 10.8) when $X$ has dimension $1$.

Suppose we have a normal crossings pair $(X,\PP)$ so $\supp(\Sigma)=\trop(X)$.  We associate to $(X,\PP)$
a polyhedral complex $\Gamma_{(X,\PP)}$ as follows: its $k$-cells are pairs
$(P,Y)$, where $P$ is a $k$-dimensional polyhedron in $\Sigma$ and $Y$ is an irreducible component of
$\CX_P$.  The cells on the boundary of $(P,Y)$ are the cells of the form $(P_i,Y_i)$,
where $P_i$ is a facet of $P$ and $Y_i$ is the irreducible
component of $\CX_{P_i}$ containing $Y$ (there is exactly one such irreducible component
because $\CX_{P_i}$ is smooth, so its irreducible components do not meet).  The complex 
$\Gamma_{(X,\PP)}$ maps naturally to $\Sigma$ by sending $(P,Y)$ to $P$.

\begin{prop} \label{prop:independence}
The underlying topological space of $\Gamma_{(X,\PP)}$ depends only on $X$.
\end{prop}
\begin{proof}
Any two polyhedral decompositions of $\trop(X)$ have a common refinement; we can further
refine this to be a normal crossings decomposition of $\trop(X)$.  It thus suffices to show that if 
$\Sigma$ and $\Sigma^{\prime}$ are normal crossings decompositions of $\trop(X)$, with associated
normal crossings pairs $(X,\PP)$ and $(X,\PP^{\prime})$, and
$\Sigma^{\prime}$ refines $\Sigma$, then the underlying topological spaces of
$\Gamma_{(X,\PP)}$ and $\Gamma_{(X,\PP^{\prime})}$ are homeomorphic.

Since $\Sigma^{\prime}$ is a refinement of $\Sigma$, we have a map $\PP^{\prime} \rightarrow \PP$.
Let $\CX^{\prime}$ be the degeneration corresponding to the pair $(X,\PP^{\prime})$.
If $P$ is a polyhedron of $\Sigma$, and $P^{\prime}$ is a polyhedron of $\Sigma^{\prime}$
contained in $P$, then this map induces a map of $\CX^{\prime}_{P^{\prime}}$ to $\CX_P$.
In particular, for every pair $(P^{\prime},Y^{\prime})$ of $\Gamma_{X,\PP^{\prime}}$,
the image of $Y^{\prime}$ in $\CX$ is contained in an unique irreducible component $Y$ of
$\CX_P$.  The map taking $(P^{\prime},Y^{\prime})$ to $(P,Y)$ is then a map of polyhedral
complexes $\Gamma_{(X,\PP^{\prime})} \rightarrow \Gamma_{(X,\PP)}.$

We have a commutative diagram:
\begin{diagram}
\Gamma_{(X,\PP^{\prime})} & \rTo & \trop(X) \\
\dTo & \ruTo & \\
\Gamma_{(X,\PP)} & & 
\end{diagram}
The fiber of $\pi: \Gamma_{(X,\PP)}\rightarrow \trop(X)$ over a point $w$ is in canonical
bijection with the set $C(\CX_P)$ of (geometric) connected components of $\CX_P$.  Similarly, the fiber
of $\pi^{\prime}: \Gamma_{(X,\PP^{\prime})} \rightarrow \trop(X)$ over $w$ is in canonical bijection
with $C(\CX^{\prime}_{P^{\prime}})$.  By Lemma~\ref{lemma:comp}, both of these sets of connected components
are in bijection with $C(\ini_w X)$; in fact, we have a commutative diagram:
\begin{diagram}
C(\ini_w X) & \rTo & (\pi^{\prime})^{-1}(w)\\ 
\dEq & & \dTo\\
C(\ini_w X) & \rTo & \pi^{-1}(w)
\end{diagram}
Thus the map $\Gamma_{(X,\PP^{\prime})} \rightarrow \Gamma_{(X,\PP)}$
is bijective, and is therefore a homeomorphism on the underlying topological
spaces of $\Gamma_{(X,\PP^{\prime})}$ and $\Gamma_{(X,\PP)}$.
\end{proof}

In light of this proposition, we denote by $\Gamma_X$ the underlying topological
space of $\Gamma_{(X,\PP)}$ for {\em any} normal crossings pair $(X,\PP)$.  
We think of $\Gamma_X$, together
with its natural map to $\trop(X)$, as a ``parameterized tropical variety.''
Note that $\Gamma_X$ inherits an integral affine structure by pullback from
$\trop(X)$; more precisely, for any normal crossings pair $(X,\PP)$,
the map $\Gamma_X \rightarrow \trop(X)$ is
linear on any polyhedron in $\Gamma_{(X,\PP)}$.

Note that for any $w \in \trop(X)$, the number of preimages
of $w$ in $\Gamma_X$ is equal to the number of connected components of
$\ini_w X$.  Therefore, if $\Sigma$ is a normal crossings decomposition 
of $\trop(X)$, and $w$ is in the relative interior
of a top dimensional cell $P$ of $\Sigma$, then the number of preimages of
$w$ is equal to the {\em multiplicity} of $P$ in $\trop(X)$.  This suggests
that we should give $\Gamma_X$ the structure of a weighted polyhedral complex
by giving every polyhedron on $\Gamma_X$ weight one.  

If we do this, then $\Gamma_X$ satisfies a ``balancing condition'' analogous
to the well-known balancing condition on $\trop(X)$.  Fix a normal
crossings decomposition $\Sigma$ of $\trop(X)$, with corresponding normal
crossings pair $(X,\PP)$.  Consider a polyhedron $(P,Y)$ of $\Gamma_{(X,\PP)}$
of dimension $\dim X - 1$, and let $\{(P_i,Y_i)\}$ be the top dimensional
polyhedra of $\Gamma_{(X,\PP)}$ containing $(P,Y)$.

Fix a point $w$ with rational coordinates in the relative interior of $P$, and let $V_P$
be the linear span, $\Span(P - w)$.  Similarly, for each $P_i$,
let $V_i$ be the positive span of $\Span^+(P_i - w)$.
Then $V_i/V_P$ is a ray in $\RR^n/V_P$; this collection of rays is the fan attached to
the toric variety $\PP_P$.  Let $v_i$ be the smallest integer vector along the ray $V_i/V_P$.

\begin{prop} \label{prop:balancing}
The $v_i$'s satisfy the ``balancing property'':
$$\sum_{(P_i,Y_i)} v_i = 0.$$
\end{prop}
\begin{proof}
Torus-equivariant rational functions on $\PP_P$ correspond to lattice vectors $u$ in
the space $(\RR^n/V_P)^*$ dual to $\RR^n/V_P$.  The valuation of $u$ along the divisor
of $\PP_P$ corresponding to $v_i$ is simply $u(v_i)$.

Now restrict $u$ to the curve $\CX_P$.  For any polyhedron $P^{\prime}$ of $\Sigma$ containing
$P$, $\CX_P$ intersects the boundary divisior $\PP_{P^{\prime}}$ in one point for each
cell $(P_i,Y_i)$ of $\Gamma_{(X,\PP)}$ with $P_i = P^{\prime}$.
The divisor of $u$ is therefore equal to
$$\sum_{(P_i,Y_i)} u(v_i) Y_i,$$
as $\CX_P$ intersects each boundary divisor $\PP_{P_i}$ transversely.
This divisor is a principal divisor and thus has degree zero.
\end{proof}

We have thus attached to any sch\"on subvariety $X$ of $\CT$, a canonical, multiplicity
free parameterization by the topological space $\Gamma_X$.  Moreover, this construction is 
functorial:  let $\CT$ and $\CT^{\prime}$ be tori over $\OO$, and let
$T$ and $T^{\prime}$ be their general fibers.  Suppose we have sch\"on subvarieties
$X$ and $Y$ of $T$ and $T^{\prime}$, respectively, and a homomorphism of tori
$T \rightarrow T^{\prime}$ that takes $X$ to $Y$.  We then have a natural map
$f:\trop(X) \rightarrow \trop(Y)$.

\begin{prop} \label{prop:functoriality}
There is a natural map
$\Gamma_X \rightarrow \Gamma_Y$ that makes the diagram
\begin{diagram}
\Gamma_X & \rTo & \Gamma_Y\\
\dTo & & \dTo\\
\trop(X) & \rTo & \trop(Y)
\end{diagram}
commute.
\end{prop}
\begin{proof}
Let $\Sigma^{\prime}$ be a normal crossings decomposisition of $\trop(Y)$. 
By proposition~\ref{prop:degeneration}
we can find a normal crossings decomposition $\Sigma$ of $\trop(X)$ 
such that the image of any cell of $\Sigma$ under the map $f$ 
is contained in a cell of $\Sigma^{\prime}$.
Let $(X,\PP)$ and $(Y,\PP^{\prime})$ be the tropical pairs corresponding to
$\Sigma$ and $\Sigma^{\prime}$, and let $\CX$ and $\CY$ denote the associated
tropical degenerations.  Since each cell of $\Sigma$ maps into a cell of
$\Sigma^{\prime}$, we obtain a map from $\CX$ to $\CY$ extending the map
$X \rightarrow Y$.

Now let $P$ be a polyhedron in $\Sigma$, and $P^{\prime}$ be the polyhedron
of $\Sigma^{\prime}$ containing the image of $P$.  Then our map
$\CX \rightarrow \CY$ induces a map $\CX_P \rightarrow \CY_{P^{\prime}}$.

If $(P,X_i)$ is a polyhedron of $\Gamma_{(X,\PP)}$, then by definition
$X_i$ is a connected component of $\CX_P$.  The image of $X_i$
in $\CY_{P^{\prime}}$ is contained in a connected component $Y_i$ of
$\CY_{P^{\prime}}$.  We can thus construct a map of polyhedral complexes
$$\Gamma_{(X,\PP)} \rightarrow \Gamma_{(Y,\PP^{\prime})}$$
that maps $(P,X_i)$ to $(P^{\prime},Y_i)$ by the map $P \rightarrow P^{\prime}$.
The induced map $\Gamma_X \rightarrow \Gamma_Y$ on underlying topological
spaces is clearly continuous and makes the diagram commute.  

To see that it is independent of choices, let $\pi_X$ and $\pi_Y$
be the projections of $\Gamma_X$ and $\Gamma_Y$ to $\trop(X)$ and $\trop(Y)$
respectively.  We then have canonical bijections between
$\pi_X^{-1}(w)$ and $C(\ini_w X)$, and between $\pi_Y^{-1}(f(w))$ 
and $C(\ini_{f(w)} Y)$.  The map $X \rightarrow Y$ induces a natural map
$\ini_w X \rightarrow \ini_{f(w)} Y$, and the diagram
\begin{diagram}
C(\ini_w X) & \rTo & \pi_X^{-1}(w)\\
\dTo & & \dTo\\
C(\ini_{f(w)} Y) & \rTo & \pi_Y^{-1}(f(w))
\end{diagram}
commutes.  As the left hand side is independent of the choices of $\Sigma$
and $\Sigma^{\prime}$, the result follows.
\end{proof}

\begin{rem} \rm Although Proposition~\ref{prop:functoriality} is stated for
maps $X \rightarrow Y$ that are {\em monomial morphisms} (i.e., that arise from
morphisms of the ambient tori), we can avoid this issue if $X$ and $Y$
are intrinsically embedded.  Recall that $X$ is {\em very affine} if it can
be embedded as a closed subscheme of a torus $T$.  In this case (c.f.
~\cite{Tevelev}, section 3) there is an intrinsic torus $T_X$ associated to $X$
a {\em canonical} embedding of $X$ in $T_X$.  Moreover, if $X$ and $Y$ are
very affine and $f: X \rightarrow Y$ is a morphism, there is a morphism
of tori $T_X \rightarrow T_Y$ that induces $f$.
\end{rem}

We also record, for later use, the following result relating the cohomology of
$\Gamma_X$ to that of $\trop(X)$:

\begin{lemma} \label{lemma:cohomology}
Let $X$ be sch\"on, and let $\Sigma$ be a normal crossings decomposition of $\trop(X)$.
Suppose that for each polyhedron $P$ in $\Sigma$, $\CX_P$ is either connected or
has dimension zero.  Then the natural
map 
$$H^r(\trop(X),\ZZ) \rightarrow H^r(\Gamma_X,\ZZ)$$
is an isomorphism for $0 \leq r < \dim X$, and an injection for $r = \dim X$.
\end{lemma}
\begin{proof}
Let $(X,\PP)$ be the normal crossings pair attached to $\Sigma$, so that
$\Gamma_{(X,\PP)}$ is a triangulation of $\Gamma$.  The polyhedra $P$ in $\Gamma_{(X,\PP)}$
with $\dim P < \dim X$ are, by construction, in bijection with the polyhedra in $\Sigma$  
with $\dim P < \dim X$.  Thus $\Gamma_{(X,\PP)}$ is obtained from $\Sigma$ by adding
additional top-dimensional cells; the result follows immediately.
\end{proof}

\section{Weight filtrations and the weight spectral sequence} \label{sec:weight}

Our goal will be to relate the combinatorial structure of $\Gamma_X$ to geometric
invariants of $X$.  The invariants that appear arise from Deligne's theory of weights, which
we now summarize.  Recall (c.f.~\cite{WeilII}, 1.2) that if $F$ is a finite field of order $q$, a continuous $l$-adic 
representation $\rho$ of $\Gal(F^{\sep}/F)$ has weight $r$ if all the eigenvalues of the geometric 
Frobenius of $F$ are algebraic integers $\alpha$, all of whose Galois conjugates have complex absolute 
value equal to $q^{r/2}$.  If $A$ is a finitely generated $\ZZ$-algebra, an {\'e}tale sheaf 
$\CF$ on $\Spec A$ 
has weight $r$ if for each closed point $s$ of $\Spec A$, the stalk $\CF_s$ has weight $r$ when 
considered as a $\Gal(k(s)^{\sep}/k(s))$-module.

Following Ito (\cite{Ito}, 2.2), we extend this definition to the case where $F$
is a purely inseparable extension of a finitely generated extension of $\FF_p$ or $\QQ$.  For
such $F$, one can find a finitely generated $\ZZ$-subalgebra $A$ of $F$ such that $F$ is a purely
inseparable extension of the field of fractions of $A$.  

In this setting, a representation $\rho$ of $\Gal(F^{\sep}/F)$ has weight $r$ if there is an 
open subset $U$ of $\Spec A$, and a smooth $\CF$ on $U$ of weight $r$, such that $\rho$ arises from $\CF$
by pullback to $\Spec F$.  The Weil conjectures imply that for any proper smooth
variety $X$ over $F$, and any $l$ prime to the characteristic of $F$,
$H^r_{\et}(X_{F^{\sep}},\QQ_l)$ has weight $r$. 

We henceforth assume that the residue field $k$ of $\OO$ is a purely inseparable
extension of a finitely generated extension of $\FF_p$ or $\QQ$.  
We also fix an $l$ prime to the characteristic of $k$.

Let $G$ be the absolute Galois group of the field $K$.  Then $G$ admits a
surjection $G \rightarrow \Gal(k^{\sep}/k)$, whose kernel is the inertia group $I_K$
of $K$.  If $M$ is a $G$-module on which $I$ acts through a finite quotient,
there is a finite index subgroup $H$ in $G$ such that $H \cap I$ acts trivially
on $M$.  Thus $\Gal(k^{\sep}/k^{\prime})$ acts on $M$ for some finite extension
$k^{\prime}$ of $k$.  We say $M$ is pure of weight $r$ if it has weight $r$ as a
$\Gal(k^{\sep}/k^{\prime})$-module.  Note that this is independent of $k^{\prime}$.

The {\'e}tale cohomology of a variety over $K$ with semistable reduction has a filtration by subquotients
which are pure in the above sense.  More precisely, let $\CX$ be a proper scheme over $\OO$, of relative 
dimension $n$,
whose fiber $X_K$ over $\Spec K$ is smooth and whose fiber $\CX_k$ over $\Spec k$ is a divisor with
simple normal crossings.  Then the Rapoport-Zink weight spectral sequence relates the {\'e}tale cohomology 
of $X_{K^{\sep}}$ to the geometry of the special fiber $\CX_k$.  More
precisely, let $\CX_{k^{\sep}}^{(r)}$ denote the disjoint union of $(r+1)$-fold intersections
of irreducible components of $\CX_{k^{\sep}}$; it is smooth of dimension $n-r$ over $k^{\sep}$.
We then have:

\begin{thm}[\cite{RZ}, Satz 2.10; see also~\cite{Ito}]
There is a spectral sequence:
$$E_1^{-r,w+r} = \bigoplus_{s \geq \max(0,-r)} H^{w-r-2s}_{\et}(\CX_{k^{\sep}}^{(2s + r)}, \QQ_l(-r-s))
\Rightarrow H^w_{\et}(X_{K^{\sep}}, \QQ_l).$$
\end{thm}

Here $\QQ_l(n)$ is the $n$th ``Tate twist'' of the constant sheaf $\QQ_l$; that is, it is the
tensor product of $\QQ_l$ with the $n$th tensor power of the sheaf $\ZZ_l(1)$, where
$\ZZ_l(1)$ is the inverse limit of the sheaves $\mu_{l^n}$ of $l$-power roots of unity.  Note that $\ZZ_l(1)$
is pure of weight $-2$.

The boundary maps of this spectral sequence are completely explicit, and can be described as follows: 
up to sign, they are direct sums of restriction maps
$$H^i_{\et}(Y, \QQ_l(-m)) \rightarrow H^i_{\et}(Y^{\prime}, \QQ_l(-m))$$
where $Y$ is an irreducible component of $\CX_{k^{\sep}}^{(j)}$ and
$Y^{\prime}$ is an irreducible component of $\CX_{k^{\sep}}^{(j+1)}$ contained in $Y$,
or Gysin maps
$$H^i_{\et}(Y^{\prime}, \QQ_l(-m)) \rightarrow H^{i+2}_{\et}(Y, \QQ_l(-m+1))$$
where $Y$ and $Y^{\prime}$ are as above.

More precisely, each term $E_1^{p,q}$ is a direct sum of terms of the form
$H^i_{\et}(Y, \QQ_l(-m))$ for some irreducible component $Y$ of $\CX_{k^{\sep}}^{(j)}$.
If $Y^{\prime}$ is an irreducible component of $\CX_{k^{\sep}}^{(j+1)}$, then we have:
\begin{itemize}
\item Whenever $H^i_{\et}(Y, \QQ_l(-m))$ is a direct summand of $E_1^{p,q}$, and
$H^i_{\et}(Y^{\prime}, \QQ_l(-m))$ is a direct summand of $E_1^{p+1,q}$, then
the corresponding direct summand of the boundary map $E_1^{p,q} \rightarrow E_1^{p+1,q}$
is (up to sign) the restriction
$$H^i_{\et}(Y, \QQ_l(-m)) \rightarrow H^i_{\et}(Y^{\prime}, \QQ_l(-m)).$$
\item Whenever $H^i_{\et}(Y^{\prime}, \QQ_l(-m))$ is a direct summand of $E_1^{p,q}$, and
$H^i_{\et}(Y, \QQ_l(-m+1))$ is a direct summand of $E_1^{p+1,q}$, then
the corresponding direct summand of the boundary map $E_1^{p,q} \rightarrow E_1^{p+1,q}$
is (up to sign) the Gysin map
$$H^i_{\et}(Y^{\prime}, \QQ_l(-m)) \rightarrow H^{i+2}_{\et}(Y, \QQ_l(-m+1)).$$
\end{itemize}

We refer the reader to example \ref{ex:curves} for a description of the weight spectral
sequence in the case when $X$ is a smooth curve.

Note that the term $E_1^{-r,w+r}$ of the weight spectral sequence is pure of weight $w+r$.
As the only map between $\QQ_l$-sheaves that are pure of different weights is the zero map,
this implies that the weight spectral sequence degenerates at $E_2$.  Moreover, the
successive quotients of the filtration on $H^*_{\et}(X_{K^{\sep}}, \QQ_l)$ induced
by the weight spectral sequence are pure.  The filtration arising in this way is called
the {\em weight filtration} on $H^*_{\et}(X_{K^{\sep}}, \QQ_l)$.

We say a $G$-module $M$ is {\em mixed} if $M$ admits an increasing $G$-stable filtration
$$\dots \subset W_rM \subset W_{r+1}M \subset \dots$$
such that $W_rM/W_{r-1}M$ has weight $r$ for all $r$.  (Such a filtration, if it exists, will
be unique.)  We say $M$ is mixed of weights between $r$ and $r^{\prime}$ if $M$ is mixed
and the quotients $W_iM/W_{i-1}M$ are nonzero only when $r \leq i \leq r^{\prime}$.
The above result shows that the cohomology of any scheme over $K$ with semistable reduction
is mixed.  More generally, one has:

\begin{thm} (c.f.~\cite{Ito}, 2.3) Let $X$ be a smooth, proper $n$-dimensional variety over 
$K$.  Then $H^r_{\et}(X_{K^{\sep}}, \QQ_l)$ is mixed of weights between
$\max(0,2r-2n)$ and $\min(2n,2r)$.
\end{thm}
\begin{proof}
If $X$ has strictly semistable reduction, i.e., $X$ is isomorphic to the general fiber of a 
scheme $\CX$ that is proper over $\OO$, and whose special fiber is a divisor with 
simple normal crossings, then this follows from the weight spectral sequence.  
The general case follows by de Jong's theory of alterations~\cite{dJ}.
\end{proof}

\begin{prop} \label{prop:cohomology}
Let $X$ be a smooth $n$-dimensional variety over $K$, and $\overline{X}$ a compactification
of $X$ such that $\overline{X} - X$ is a divisor with simple normal crossings.  Then for $r \leq n$,
$H^r_{\et}(X_{K^{\sep}}, \QQ_l)$ is mixed of weights between $0$ and $2r$, and the natural map
$$W_0H^r_{\et}(\overline{X}_{K^{\sep}}, \QQ_l) \rightarrow W_0H^r_{\et}(X_{K^{\sep}}, \QQ_l)$$
is an isomorphism.
\end{prop}
\begin{proof}
Let $D$ be the divisor $\overline{X} \setminus X$, and let $\overline{D}_1, \dots, \overline{D}_r$ be 
its irreducible components.  Let $X_i$ be the open subset $X \setminus 
\{\overline{D}_1 \cup \dots \cup \overline{D}_i\}$.  We proceed
by induction on $i$; the case $i = 0$ is clear. 

Suppose the proposition is true for $i$.  Define 
$$D_i = \overline{D}_{i+1} \setminus \{\overline{D}_1 \cup \dots \cup \overline{D}_i\},$$
so that $X_i \setminus X_{i+1} = D_i$.
By the inductive hypothesis the spaces $H^r_{\et}((X_i)_L, \QQ_l)$
and $H^r_{\et}((D_i)_L, \QQ_l)$
are mixed of weights between $0$ and $2r$ for $r \leq n$.  We have a Gysin sequence:
$$
\begin{array}{cccccc}
H^{r-2}_{\et}((D_i)_L, \QQ_l(-1)) & \rightarrow &
H^r_{\et}((X_i)_L, \QQ_l) & \rightarrow & 
H^r_{\et}((X_{i+1})_L, \QQ_l) & \rightarrow \\ 
H^{r-1}_{\et}((D_i)_L, \QQ_l(-1)) & \rightarrow & \dots,
\end{array}
$$
and the first and last terms are mixed of weights between $2$ and $2r$.  It follows that
$H^r_{\et}((X_{i+1})_L, \QQ_l)$ is mixed of weights between $0$ and $2r$ as
required.  We also obtain an isomorphism
$$W_0H^r_{\et}((X_i)_L, \QQ_l) \cong W_0H^r_{\et}((X_{i+1})_L, \QQ_l)$$ 
and hence by induction the desired isomorphism
$$W_0H^r_{\et}(\overline{X}_L, \QQ_l) \cong W_0H^r_{\et}(X_L, \QQ_l).$$
\end{proof}

\section{Cohomology of sch\"on varieties} \label{sec:schon}

The control that tropical geometry gives over the degenerations of sch\"on subvarieties
$X$ of $T$ has significant consequences on the level of cohomology.  In particular the theory
of vanishing cycles allows one to relate the {\'e}tale cohomology of a nice tropical compactification 
of $X$ to that of its tropical degeneration.  When the degeneration is a divisor with
simple normal crossings, this relationship is given by the Rapoport-Zink weight spectral
sequence. 

We apply this sequence in the setting of tropical geometry.  Let $X$ be sch\"on.  By
Proposition~\ref{prop:degeneration} there is a polyhedral complex $\Sigma$, with
support equal to $\trop(X)$ and corresponding toric scheme $\PP$, such that the pair 
$(X,\PP)$ is tropical, the corresponding compactification $\overline{X}$ of $X$
is smooth with simple normal crossings boundary, and the special fiber of the corresponding
tropical degeneration $\CX$ is a divisor with simple normal crossings.

The polyhedral complex $\Gamma_{(X,\PP)}$ encodes the combinatorics of the special fiber
$\CX_{k^{\sep}}$.  In particular $\CX_{k^{\sep}}$ is a union of smooth 
connected varieties $\CX_v$, where $v$ runs over the vertices of $\Gamma_{(X,\PP)}$.  The varieties
$\CX_{v_1}, \dots, \CX_{v_r}$ meet if and only if $v_1, \dots, v_r$ are the vertices of
a polyhedron in $\Gamma_{(X,\PP)}$.  
[Note that since $\CX_k$ is a simple normal crossings divisor, if $Y_0, \dots, Y_r$ intersect
in codimension $r$ then they are the only irreducible components of $\CX_{k^{\sep}}$ containing 
their intersection.]  

We have a natural map $\Gamma_{(X,\PP)} \rightarrow \Sigma$. 
Since $\Sigma$ is a triangulation of $\trop(X)$, and $\Gamma_{(X,\PP)}$ is a triangulation
of $\Gamma_X$, this induces a natural map
$$H^r(\trop X, \QQ_l) \rightarrow H^r(\Gamma_X, \QQ_l).$$
By the proof of Lemma \ref{lemma:cohomology}, this map is an isomorphism if $\CX_P$ is connected for
every polyhedron $P$ in $\Sigma$, or (equivalently) if $\ini_w X$ is connected
for every $w$ in $\trop(X)$. 

\begin{thm} \label{thm:main} There is a natural isomorphism 
$$H^r(\Gamma_X, \QQ_l) \cong W_0 H^r_{\et}(\overline{X}_{K^{\sep}},\QQ_l),$$
and hence a natural map
$$H^r(\trop(X), \QQ_l) \rightarrow W_0 H^r_{\et}(\overline{X}_{K^{\sep}},\QQ_l).$$
This map is an isomorphism if $\CX_P$ is connected for
every polyhedron $P$ in $\Sigma$.
\end{thm}
\begin{proof}
The bottom nonzero row of the $E_1$ term of the Rapoport-Zink spectral sequence (i.e., the $w=-r$ row)
is the complex:
$$
H^0_{\et}(\CX_{k^{\sep}}^{(0)},\QQ_l) \rightarrow
H^0_{\et}(\CX_{k^{\sep}}^{(1)},\QQ_l) \rightarrow
H^0_{\et}(\CX_{k^{\sep}}^{(2)},\QQ_l) \rightarrow \dots
$$
in which the horizontal maps are restriction maps.
This is simply the coboundary complex of the polyhedral complex formed by the bounded 
cells of $\Gamma_{(X,\PP)}$.  This polyhedral complex is homotopy equivalent to 
$\Gamma_{(X,\PP)}$.  We thus have a natural isomorphism
$$E_2^{r,0} \cong H^r(\Gamma_X, \QQ_l).$$
\end{proof}

\begin{rem} \rm Theorem~\ref{thm:main} shows in particular that the space
$W_0 H^r_{\et}(\overline{X}_{K^{\sep}},\QQ_l)$, which {\em a priori} depends on
$\overline{X}$ and thus a choice of $\Sigma$, in fact depends only on $X$ and is
{\em independent} of $\Sigma$.  Proposition~\ref{prop:cohomology} establishes this 
directly on the level of cohomology.
\end{rem}

The above results allow us to translate results about the cohomology of complete intersections 
in toric varieties into results about their tropicalizations.  For instance:

\begin{cor} \label{cor:CI}
Let $X$ be a sch\"on subvariety of $T$, and $\PP_K$ a smooth projective toric variety of $T$ such that:
\begin{enumerate}
\item the closure $Z$ of $X$ in $\PP_K$ is a smooth complete intersection of ample divisors, and
\item the boundary $Z \setminus X$ is a divisor with simple normal crossings.
\end{enumerate}
Then $H^r(\Gamma_X, \QQ_l) = 0$ for $1 \leq r < \dim X$.
\end{cor}
\begin{proof}
By Proposition~\ref{prop:degeneration} and Theorem~\ref{thm:main} there is a tropical
pair $(X,\PP^{\prime})$, with corresponding compactification $\overline{X}$ of $X$, such that
$H^r(\Gamma_X, \QQ_l)$ is isomorphic to $W_0 H^r_{\et}(\overline{X}_{K^{\sep}},\QQ_l)$.
By Proposition~\ref{prop:cohomology} the latter is isomorphic to
$W_0 H^r_{\et}(Z_{K^{\sep}},\QQ_l)$.  

Since $Z$ is a complete intersection in $\PP_K$, the Lefschetz hyperplane theorem shows that
for $r < \dim X$, the restriction map
$$H^r_{\et}(\PP_{K^{\sep}},\QQ_l) \rightarrow H^r_{\et}(Z_{K^{\sep}},\QQ_l)$$ 
is an isomorphism.  But $\PP_K$ is a smooth toric variety, and hence has good reduction.
The weight spectral sequence thus shows that $W_0 H^r_{\et}(\PP_{K^{\sep}},\QQ_l) = 0$
for $r > 0$.
\end{proof}

Under more restrictive hypotheses on $X$, we can turn the above result into a result about the
cohomology of $\trop(X)$.  This will be the main goal of section~\ref{sec:CI}.

\section{Monodromy} \label{sec:monodromy}

In many situations, the weight filtration has an alternative interpretation in terms of monodromy.
Let $X$ be a variety over $K$, and consider the base change $X_{K^{\sep}}$ of $X$ to $K^{\sep}$.
The group $\Gal(K^{\sep}/K)$ admits a map to $\Gal(k^{\sep}/k)$; the kernel is the
inertia group $I_K$.  The group $I_K$ is profinite; if $l$ is prime to the characteristic of $k$
then the pro-$l$ part $I_K^{(l)}$ of $I_K$ is isomorphic to $\ZZ_l(1)$.  (The Tate twist here refers
to the fact that the quotient $\Gal(k^{\sep}/k)$ acts on $I_K^{(l)}$ by conjugation in the
same way that it acts on the inverse limit of the roots of unity $\mu_{l^n}$.

The group $\Gal(K^{\sep}/K)$ acts on the {\'e}tale cohomology
$H^i_{\et}(X_{K^{\sep}},\QQ_l)$ for any prime $l$.  This action is quasi-unipotent, i.e.
a subgroup of $H$ of $I_K$ of finite index acts unipotently on $H^i_{\et}(X_{K^{\sep}},\QQ_l)$.
(And thus the action of $H$ factors through $I_K^{(l)}$.)
In particular there is a nilpotent map 
$$N: H^i_{\et}(X_{K^{\sep}},\QQ_l) \rightarrow H^i_{\et}(X_{K^{\sep}},\QQ_l(-1))$$
called the {\em monodromy operator} such that for all $\sigma \in H$, $\sigma$ acts on
$H^i_{\et}(X_{K^{\sep}},\QQ_l)$ by $\exp(t_l(\sigma)N)$, where $t_l$ is the map
$I_K \rightarrow I_K^{(l)} \cong \ZZ_l(1).$

Now, if $V$ is any finite dimensional vector space, with a nilpotent endomorphism $N$ such that $N^r = 0$, 
then there is a unique increasing filtration $\{V_i\}$ on $V$ such that:
\begin{itemize}
\item $V_r = V$,
\item $V_{-r} = 0$,
\item $N$ maps $V_i$ to $V_{i-2}$, and
\item $N^i$ induces an isomorphism $V_i/V_{i-1} \rightarrow V_{-i}/V_{-i-1}$.
\end{itemize}
(see~\cite{WeilII} I, 1.7.2 for details.)  We thus obtain a natural filtration, called the
monodromy filtration, on $H^i_{\et}(X_{K^{\sep}}, \QQ_l)$.

\begin{rem} \rm If $V$ consists of a single Jordan block of dimension $r$, one sees easily that
$V_i/V_{i-1}$ is one-dimensional for $i \in \{r-1, r-3, \dots, -r+1\}$, and zero otherwise.  
Moreover, $V_{r-1-2k}$ is the image of $N^k$ for $0 \leq k \leq r-1$.
It is thus straightforward to read off the filtration coming from an arbitrary $V$ and $N$ from
a Jordan normal form for $N$.  The filtration is independent of choices, even though the Jordan
normal form of $N$ is not.
\end{rem}

When $X$ has a semistable model, one can read the monodromy action on $X$ off
of the weight spectral sequence $E^{p,q}$.  More precisely, there is a
monodromy operator
$N: E_1^{p,q} \rightarrow E_1^{p+2,q-2}(-1)$, which converges to the the monodromy operator $N$
on $H^i_{\et}(X_{K^{\sep}}, \QQ_l)$.  It is easily described:
if $H^i_{\et}(Y,\QQ_l(-m))$ occurs as a direct summand of $E_1^{p,q}$, and
$H^i_{\et}(Y,\QQ_l(-m+1))$ occurs as a direct summand of $E_1^{p+2,q-2}$, then
the corresponding direct summand of $N$ is the identity 
$$H^i_{\et}(Y,\QQ_l(-m)) \rightarrow H^i_{\et}(Y,\QQ_l(-m+1))(-1).$$
All other direct summands of $N$ are the zero map.

The following conjecture (the ``weight-monodromy conjecture'') relates the weight filtration to the 
monodromy filtration in this situation:

\begin{conj} \label{conj:w-m}
The top nonzero power of the monodromy operator:
$$N^r: E_2^{-r,w+r} \rightarrow E_2^{r,w-r}$$
is an isomorphism for all $r,w$.  In particular the weight filtration on $H^i_{\et}(X_{K^{\sep}},\QQ_l)$
$E$ coincides (up to a shift in degree) with the monodromy filtration; that is,
$$H^w_{\et}(X_{K^{\sep}},\QQ_l)_{-r}/H^w_{\et}(X_{K^{\sep}},\QQ_l)_{-r-1} \cong 
W_{w-r} H^w_{\et}(X_{K^{\sep}},\QQ_l).$$
\end{conj}

The weight-monodromy conjecture is well-known to hold for curves and surfaces.  If
$\OO$ is an equal characteristic discrete valuation ring, it is a difficult theorem of
Ito (\cite{Ito}, Theorem 1.1).  It is open in general when $\OO$ has mixed characteristic.

For the remainder of this section we assume we are in a situation where Conjecture~\ref{conj:w-m}
holds.  The following result, due to Speyer (\cite{Sp3}, Theorem 10.8) for curves, follows 
immediately:

\begin{cor} \label{cor:bound} 
Let $b_r(\Gamma_X)$ and $b_r(X)$ denote the $r$th Betti numbers of $\Gamma_X$ and $X$,
respectively.  Then we have:
$$b_r(\Gamma_X) \leq \frac{1}{r+1} b_r(X).$$
\end{cor}
\begin{proof}
Theorem~\ref{thm:main}, together with the weight-monodromy conjecture, shows that $b_r(\Gamma_X)$ 
is the dimension of $H^r_{\et}(\overline{X}_{K^{\sep}},\QQ_l)_{-r}$.  The dimension of this
piece of the monodromy filtration counts the number of Jordan blocks of size $r+1$ 
in a Jordan normal form for $N$ acting on $H^r_{\et}(\overline{X}_{K^{\sep}},\QQ_l)$.  
In particular the dimension of the latter is at least $r+1$ times the dimension of the 
former.
\end{proof}

Suppose $X_K$ is an $n$-dimensional variety.  There is a geometric interpretation of the action of the $n$th power of the monodromy operator on the middle-dimensional cohomology.
\begin{prop} \label{prop:npower} The $d$th power of the monodromy map acting on middle cohomology,
\[N^d:H^d_{\et}(X_{K^{\sep}},\QQ_l)_{d}/H^d_{\et}(X_{K^{\sep}},\QQ_l)_{d-1}\rightarrow H^d_{\et}(X_{K^{\sep}},\QQ_l)_{-d}(-d).\]
is the map
\[H_d(\Gamma_{(X,\PP)},\QQ_l)(-d)\rightarrow H^d(\Gamma_{(X,\PP)},\QQ_l)(-d)\]
induced from the ``volume pairing'' on the parameterizing complex $\Gamma_{(X,\PP)}$ which takes a pair of (integral) $d$-dimensional cycles to the (oriented) lattice volume of their intersection.  
\end{prop}

\begin{proof}
The term $H^d_{\et}(X_{K^{\sep}},\QQ_l)_{d}/H^d_{\et}(X_{K^{\sep}},\QQ_l)_{d-1}$ is computed by the $(-d,d)$-entry in the Rapoport-Zink spectral sequence.  The $d$th row is:
$$
H^0_{\et}(\CX_{k^{\sep}}^{(d)},\QQ_l)(-d) \rightarrow
H^2_{\et}(\CX_{k^{\sep}}^{(d-1)},\QQ_l)(-d+1) \rightarrow\dots\rightarrow
H^{2d}_{\et}(\CX_{k^{\sep}}^{(0)},\QQ_l)$$
where the horizontal map is the Gysin map of $\CX^{(k)}\rightarrow \CX^{(k-1)}$.  Since each component of $\CX^{(k)}$ is an $(d-k)$-dimensional smooth variety, this is the chain complex formed by the bounded cells of $\Gamma_{(X,\PP)}$.  Consequently, $E_2^{-k,d}\cong H_k(\Gamma,\QQ_l)(-k)$.  Now, 
\[N^d:E_2^{-d,d}\cong H_d(\Gamma,\QQ_l)(-d)\rightarrow E_2^{0,0}(-d)\cong H^d(\Gamma_{(X,\PP)},\QQ_l)(-d)\] is induced from the identity map on $H^0_{\et}(\CX_{k^{\sep}}^{(d)},\QQ_l)$.  In the language of homology and cohomology of $\Gamma_{(X,\PP)}$, it comes from the 
map $C_d(\Gamma_{(X,\PP)})\rightarrow C^d(\Gamma_{(X,\PP)})$ taking a simplex $F$ to the cocycle $\delta_F:C^d(\Gamma_{(X,\PP)})\rightarrow\ZZ$ that is the indicator function of $F$.  Consequently, if we view $N^d$ as  a bilinear pairing on $H_d(\Gamma,\QQ_l)$, it is the volume pairing as every bounded top-dimensional cell of $\Gamma_{(X,\PP)}$ has volume $1$.
\end{proof}

\begin{ex} \label{ex:curves}
\rm Suppose that $X_K$ is a curve of genus $g$.  Then $\CX_{k^{\sep}}^{(0)}$ is the 
normalization of $\CX_{k^{\sep}}$; it is a disjoint union of smooth curves $C_i$ of genus 
$g_i$.  On the other hand, $\CX_{k^{\sep}}^{(1)}$ is the set of singular points of 
$\CX_{k^{\sep}}$; each such point lies on exactly two of the $C_i$.  The corresponding
weight spectral sequence is nonzero only for $-1 \leq r \leq 1$ and $0 \leq w+r \leq 2$; it 
looks like:
$$
\begin{array}{ccccc}
\bigoplus_{p \in \CX_{k^{\sep}}^{(1)}} \QQ_l(-1) & \rightarrow & \bigoplus_i H^2_{\et}(C_i, \QQ_l) & & 0\\
0 & & \bigoplus_i H^1_{\et}(C_i,\QQ_l) & & 0\\
0 & & \bigoplus_i H^0_{\et}(C_i,\QQ_l) & \rightarrow & \bigoplus_{p \in \CX_{k^{\sep}}^{(1)}} \QQ_l
\end{array}
$$
The sequence clearly degenerates at $E_2$.  
The monodromy operator $N$ is nonzero only from $E_1^{-1,2}$ to $E_1^{1,0}(-1)$; 
it is simply the identity map on $$\bigoplus_{p \in \CX_{k^{\sep}}^{(1)}} \QQ_l(-1).$$
We thus find that the middle quotient of
the monodromy filtration on $H^1_{\et}(X_{K^{\sep}},\QQ_l)$ is isomorphic to
the direct sum of $H^1_{\et}(C_i,\QQ_l)$, whereas the top and bottom quotients are isomorphic
to $H_1(\Gamma,\QQ_l)$, (resp. $H^1(\Gamma,\QQ_l)$) where $\Gamma$ is the dual graph of 
$\CX_{k^{\sep}}$.  As above, the map $N: H_1(\Gamma,\QQ_l) \rightarrow H^1(\Gamma,\QQ_l)$
can be interpreted as the length pairing on $H_1(\Gamma,\QQ_l)$. \end{ex}

This example has a more classical interpretation.  If we let $J$ be the Jacobian of
$\overline{X}$, then the connected component of the identity in the
special fiber of the N{\'e}ron model of $J$ over $\OO$ is an extension of an abelian variety
by a torus; let $\chi$ be the character group of this torus.  Then $\chi$
is naturally isomorphic to $H_1(\Gamma,\ZZ)$.  Moreover, one has a monodromy pairing
$\chi \times \chi \rightarrow \ZZ$ (see \cite{SGA} for details.)  If one identifies
$\chi$ with $H_1(\Gamma,\ZZ)$, the resulting pairing on $H_1(\Gamma,\ZZ)$ is precisely
the length pairing.

To summarize:
\begin{prop} If $X$ is a sch\"on open subset of a smooth proper curve $\overline{X}$ over $K$, then
the ``length pairing'' on $\Gamma_X$ coincides with the monodromy pairing on
the character group $\chi$ associated to the Jacobian of $\overline{X}$.
\end{prop}

This has connections to Mikhalkin's construction of tropical Jacobians.  Given a
tropical curve $\Gamma$, which Mikhalkin interprets as a metric graph, the length
pairing on $\Gamma$ induces a map $H_1(\Gamma) \rightarrow \Hom(H_1(\Gamma),\ZZ)$; Mikhalkin
defines the tropical Jacobian of $\Gamma$ to be the torus $\Hom(H_1(\Gamma),\RR)/H_1(\Gamma)$.
This torus has a natural integral affine structure induced from that on $\Hom(H_1(\Gamma),\RR)$.
See~\cite{MZ} for details.

Mikhalkin's definition is purely combinatorial but has a nice interpretation in
terms of the uniformization of abelian varieties: if $J$ is the Jacobian of
$\overline{X}$ then there is a pairing $\chi \times \chi \rightarrow \overline{K}^*$
whose valuation is the monodromy pairing.  This pairing gives an embedding of
$\chi$ as a lattice in the torus $\Hom(\chi, \overline{K}^*)$; the quotient
$\Hom(\chi,\overline{K}^*)/\chi$ is a rigid space isomorphic to $J$.
If we ``tropicalize'' this space by taking valuations, we obtain the
space $\Hom(\chi,\RR)/\chi$, where $\chi$ embeds into $\Hom(\chi,\RR)$ by
the monodromy pairing.  In particular the ``tropicalization'' of $J$ is
the tropical Jacobian of $\Gamma_X$.

The upshot is that- provided we are careful about what we mean by tropicalization-
``tropicalization'' commutes with taking Jacobians.

The following result of \cite{KMM} is another easy consequence of this point of view:
\begin{prop}[\cite{KMM}, Theorem 6.4] Let $X$ be a sch\"on open subset of an
elliptic curve $\overline{X}$ over $K$ with potentially multiplicative reduction.
Then $H_1(\Gamma_X,\ZZ)$ is isomorphic to $\ZZ$, and valuation of the $j$-invariant 
$j(\overline{X})$ is equal to $-a$, where $a$ is the length of the unique cycle in $\Gamma_X$.
\end{prop}
\begin{proof}
Replacing $\OO$ with a suitable ramified extension we may assume that $\Gamma_X$ is integral.
Then $\overline{X}$ has split multiplicative reduction.  This base change scales both $\Gamma_X$ and
the valuation of $j(\overline{X})$ by the degree $d$ of the extension.  Now the tropical
degeneration $\CX$ of $X$ associated to $\Gamma_X$ gives a model of $\overline{X}$ whose special
fiber contains a cycle of rational curves, of length equal to the lattice length $a$ of the unique
cycle in $\Gamma_X$.  The conductor-discriminant formula (\cite{Si}, Theorem 11.1) then shows that
the valuation of $j(X)$ is equal to $-a$.
\end{proof}

In fact, it is easy to see that any smooth curve $\overline{X}$ contains a sch\"on open subset:
take a semistable model of $\overline{X}$, embed it in $\PP^n_{\OO}$, let $\CT$ be the complement of
$n+1$ hyperplanes in general positionin $\PP^n_{\OO}$, and take $X = \CT \cap \overline{X}$.  
Then the compactification $\overline{X}$ of $X$ in $\PP^n_{\OO}$ is tropical, and one verifies
easily that the multiplication map is smooth.  Thus the above result applies to all elliptic curves
with potentially multiplicative reduction.
Luxton and Qu~\cite{LQ} have shown that any variety over a field of characteristic $0$ contains a
sch\"on open subset.

\section{Complete Intersections} \label{sec:CI}

In the constant coefficient case, a (Zariski) general hyperplane section of a sch\"on variety is
sch\"on.  Unfortunately this is no longer true in the nonconstant coefficient case.  For instance,
let $X_k$ be a singular hypersurface in $T_k$.  Then any hypersurface $X$ in $T_K$ that reduces
modulo $\pi$ to $X_k$ has $\ini_{(0,\dots,0)} X = X_k$, and hence cannot be sch\"on.  The set
of such $X$ is a rigid analytic open subset of the projective space of hypersurfaces of fixed degree.

As this example suggests, to study loci of sch\"on varieties in a nonconstant coefficient setting,
one needs to work with the rigid analytic topology
rather than the Zariski topology.  (For the basics of the theory of rigid analytic spaces we refer the
reader to~\cite{EKL} or~\cite{Schneider}; we use very little here.)  

To make precise the connection to rigid geometry, we first observe:

\begin{lemma} \label{lemma:transversality}
Let $\PP$ be a toric scheme, proper over $\OO$, and let $X$ be a subvariety of 
the open torus $T$ in $\PP \times_{\OO} \Spec K$.
Suppose that for all polyhedra $P$ in the polyhedral complex $\Sigma$ corresponding to $P$, the 
closure $\CX$ of $X$ in $\PP$ intersects $\PP_P$ transversely.  Then $X$ is sch\"on, and
$(X, \PP^{\prime})$ is a normal crossings pair, where $\PP^{\prime}$ is the open subset
of $\PP$ obtained by deleting all torus orbits that do not meet $\CX$.

Conversely, if $X$ is sch\"on and there exists a toric open subset $\PP^{\prime}$ of $\PP$
such that $(X,\PP^{\prime})$ is a normal crossings pair, then the closure of $X$ intersects $\PP_P$ 
transversely for all polyhedra $P$ in $\Sigma$.
\end{lemma}

\begin{proof}
Consider the multiplication map
$$m:\CT \times \CX \rightarrow \PP^{\prime}.$$
If $y$ is a point in $\PP^{\prime}$ in the torus orbit corresponding to a polyhedron $P$ in
the subcomplex $\Sigma^{\prime}$ of $\Sigma$ corresponding to $\PP^{\prime}$, then the fiber over 
$y$ is isomorphic to the product
$\CX \cap \PP^{\prime}_P$ with a torus.  By assumption,
this is smooth, so $m$ has smooth fibers.  The argument of~\cite{Hacking}, Lemma 2.6
then shows that $m$ is smooth.  It follows that $X$ is sch\"on and $(X,\PP^{\prime})$
is a normal crossings pair.  The converse is clear.
\end{proof}

Note that the lemma implies that $\trop(X)$ will be equal to the support of $\Sigma^{\prime}$ 
for all such $X$.  One can therefore use this result to study the space of sch\"on subvarieties of 
a toric variety over $K$ with a given tropicalization.  We will not pursue this here, beyond a few
straightforward observations.

Suppose $\PP$ is projective.  Fix an $X$ as in the lemma, and let $\Hilb(\PP)$ be the Hilbert
scheme over $\OO$ parameterizing subschemes of $\PP$ with the same Hilbert polynomial as the closure
of $X$.  Complex points of $\Hilb(\PP)$ correspond to subschemes of the special fiber of $\PP$; those 
that meet each $\PP_P$ transversely form an open subset $U_0$ of 
$\Hilb(\PP) \times_{\OO} \Spec k$.

Now if $y$ is a point of $\Hilb(\PP)(\overline{K})$, then $y$ corresponds to a subscheme $X_y$ of the 
general fiber of $\PP$ over a finite extension of $K$.  Then $X_y \cap T$ will satisfy the hypotheses
of Lemma~\ref{lemma:transversality} if, and only if, $y$ specializes to a point $y_0$
on the special fiber of $\Hilb(\PP)$ that lies in $U_0$.  The set of points
$y$ that specialize to $U_0$ forms a ``neighborhood of $U_0$'' in the rigid analytic topology
on $\Hilb(\PP).$  More precisely, let $\Hilb(\PP)^{\rig}$ denote the rigid analytic space
associated to the general fiber of $\Hilb(\PP)$; then $\Hilb(\PP)^{\rig}$ is equipped
with a ``reduction mod $\pi$'' map
$$\Hilb(\PP)^{\rig} \rightarrow \Hilb(\PP) \times_{\OO} \Spec k.$$
The preimage of $U_0$ under this map is an admissible open subset $U^{\rig}$ of
$\Hilb(\PP)^{\rig}$, and those $y \in \Hilb(\PP)(\overline{K})$ such that $X_y \cap T$ satisfies the 
hypotheses of Lemma~\ref{lemma:transversality} are precisely the $\overline{K}$-points of $U^{\rig}$.

If we restrict our attention to complete intersections, we can say more than this.
In particular fix a projective toric scheme $\PP$ over $\OO$, and ample line
bundles $L_1, \dots, L_s$ on $\PP$.  The space $\CH$ parameterizing tuples
$(D_1, \dots, D_s)$ such that for each $i$, $D_i$ is an effective divisor in the linear system
corresponding to $L_i$, and all the $D_i$'s intersect transversely, is an open subset
of a product of projective spaces over $\OO$.  

By Bertini's theorem, the set of points in $\CH(k)$ that correspond to divisors
$(D_1, \dots, D_s)$ in $\PP \times_{\OO} \Spec k$ such that $D_1 \cap \dots \cap D_s$ 
intersects each stratum $\PP_P$ of $\PP$ transversely is an open dense subset $U_0$ of
the special fiber of $\CH$.  The preimage of $U_0$ under the reduction map
$$\CH^{\rig} \rightarrow \CH \times_{\OO} \Spec k$$
is a (necessarily nonempty) admissible open subset $U^{\rig}$ of $\CH^{\rig}$; the
points of $U^{\rig}$ correspond precisely to those complete intersections
$D_1 \cap \dots \cap D_s$ whose intersection with $T$ satisfies the conditions of 
Lemma~\ref{lemma:transversality}.

Moreover, if $(D_1, \dots, D_s)$ is a $K$-point of $U^{\rig}$, and $X$ is the corresponding
complete intersection $D_1 \cap \dots \cap D_s \cap T$ in $T$, then for each polyhedron $P$ in
$\Sigma$, $\CX_P = D_1 \cap \dots \cap D_s \cap \PP_P$ is the intersection of ample divisors
in the smooth toric variety $\PP_P$, and is therefore either zero-dimensional or connected.

Lemma~\ref{lemma:cohomology} and Theorem~\ref{thm:main} now have immediate implications
for the cohomology of $\trop(X)$:

\begin{thm} \label{thm:CI}
Let $(D_1, \dots, D_s)$ be a $K$-point of $U^{\rig}$, and set
$$X = D_1 \cap \dots \cap D_s \cap T.$$
Then $H^r(\trop(X),\QQ_l)$ vanishes for $1 \leq r < \dim X$, and the natural map:
$$H^r(\trop(X),\QQ_l) \rightarrow W_0 H^r_{\et}(\overline{X}_{K^{\sep}},\QQ_l)$$
is injective for $r = \dim X$.
\end{thm}
\begin{proof}
The above discussion shows that $X$ is sch\"on and $(X,\PP)$ is a normal crossings pair.
We thus apply Theorem~\ref{thm:main} and Lemma~\ref{lemma:cohomology} to see that the
map
$$H^r(\trop(X),\QQ_l) \rightarrow W_0 H^r_{\et}(\overline{X}_{K^{\sep}},\QQ_l)$$
is an isomorphism for $0 \leq r < \dim X$ and injective for $r = \dim X$.  On the other
hand, $\overline{X}$ is a complete intersection in the general fiber of the smooth toric
variety $\PP \times_{\OO} \Spec K$.  The result thus follows from Corollary~\ref{cor:CI}.
\end{proof}


\end{document}